\documentclass[10pt]{elsarticle}

\usepackage{amsfonts,amsmath,amsthm,amscd,amssymb,latexsym,mathrsfs,enumitem,tikz}
\usepackage[english]{babel}
\usetikzlibrary{positioning,arrows,shapes,shadows}

\theoremstyle{plain}

\newtheorem{theorem}{Theorem}
\newtheorem{lemma}[theorem]{Lemma}
\newtheorem{proposition}[theorem]{Proposition}
\newtheorem{corollary}[theorem]{Corollary}
\newtheorem{question}{Question}
\newdefinition{remark}{Remark}
\newdefinition{definition}{Definition}
\newdefinition{example}{Example}
\newproof{pf}{Proof}

\newtheorem{thmm}{Theorem}

\numberwithin{equation}{section}

\begin{document}

\title{Extending Baire-one functions on compact spaces}
\author{Olena Karlova}
\ead{maslenizza.ua@gmail.com}
\author{Volodymyr Mykhaylyuk}
\ead{vmykhalyuk@ukr.net}

\address{Department of Mathematical Analysis, Faculty of Mathematics and Informatics, Yurii Fedkovych Chernivtsi National University, Kotsyubyns'koho str., 2, Chernivtsi, Ukraine}

\begin{abstract}
We answer a question of O. Kalenda and J. Spurn\'{y} from~\cite{Ka} and give an example of a completely regular hereditarily Baire  space $X$ and a Baire-one function $f:X\to [0,1]$ which   can not be extended to a Baire-one function on $\beta X$.
\end{abstract}


\maketitle

\section{Introduction}
The classical Kuratowski's extension theorem~\cite[35.VI]{Kur:Top:Eng1} states that any map $f:E\to Y$ of the first Borel class to a Polish space $Y$ can be extended to a map $g:X\to Y$ of the first Borel class if $E$ is a $G_\delta$-subspace of a metrizable space $X$. Non-separable version of Kuratowski's theorem was proved by Hansell~\cite[Theorem 9]{Hansell:1974}, while abstract topological versions of Kuratowski's theorem were developed in~\cite{Hol,Ka,Karlova:CMUC:2013}.

Recall that a map $f:X\to Y$ between topological spaces $X$ and $Y$ is said to be
\begin{itemize}[topsep=2pt, itemsep=-2pt]
 \item[-] {\it  Baire-one}, $f\in {\rm B}_1(X,Y)$, if it is a pointwise limit of a sequence of continuous maps $f_n:X\to Y$;

 \item[-] {\it functionally $F_\sigma$-measurable} or {\it of the first functional Borel class}, $f\in {\rm K}_1(X,Y)$, if the preimage $f^{-1}(V)$ of any open set $V\subseteq Y$ is a union of a sequence of zero sets in $X$.
\end{itemize}
Notice that every functionally $F_\sigma$-measurable map belongs to the first Borel class for any $X$ and $Y$; the converse inclusion is true for perfectly  normal $X$; moreover, for a topological space $X$ and a metrizable  separable connected and locally path-connected space $Y$ we have the equality ${\rm B}_1(X,Y)={\rm K}_1(X,Y)$~(see \cite{KM:Umzh}).

Kalenda and Spurn\'{y}  proved the following result  \cite[Theorem 13]{Ka}.
\begin{thmm}
  Let $E$ be a Lindel\"{o}f hereditarily Baire subset of a completely
regular space $X$ and $f:E\to\mathbb R$ be a Baire-one function. Then there exists a Baire-one function $g:X\to\mathbb R$ such that $g=f$ on $E$.
\end{thmm}

The  simple example shows that the assumption that $E$ is  hereditarily   Baire cannot be omitted: if $A$ and $B$ are disjoint dense subsets of $E=\mathbb Q\cap [0,1]$ such that $E=A\cup B$ and $X=[0,1]$ or $X=\beta E$, then the characteristic function $f=\chi_A:E\to\mathbb R$ can not be extended to a Baire-one function on $X$. In connection with this the following question was formulated in \cite[Question 1]{Ka}.

\begin{question}
Let $X$ be a hereditarily Baire completely regular space and $f$ a
Baire-one function on $X$. Can $f$ be extended to a Baire-one function on $\beta X$?
\end{question}

We answer the question of Kalenda and Spurn\'{y} in negative. We introduce a notion of functionally countably fragmented map (see definitions  in Section~\ref{s22}) and prove that for a Baire-one function $f:X\to \mathbb R$ on a completely regular space $X$ the following conditions are equivalent:
(i) $f$ is functionally countably fragmented; (ii) $f$ can be extended to a Baire-one function on $\beta X$. In Section~\ref{s33} we give an example of a completely regular hereditarily Baire (even scattered) space $X$ and a Baire-one function $f:X\to [0,1]$ which is not functionally countably fragmented and consequently can not be extended to a Baire-one function on $\beta X$.

\section{Extension of countably fragmented functions}\label{s22}

 Let  $X$ be a topological space and $(Y,d)$ be a metric space. A map $f:X\to Y$ is called {\it $\varepsilon$-fragmented} for some $\varepsilon>0$  if for every closed nonempty set  $F\subseteq X$ there exists a nonempty relatively open set $U\subseteq F$ such that ${\rm diam}f(U)<\varepsilon$. If $f$ is $\varepsilon$-fragmented for every  $\varepsilon>0$, then it is called  {\it fragmented}.

Let $\mathscr U=(U_\xi:\xi\in[0,\alpha])$ be  a transfinite sequence of subsets of a topological space $X$. Following~\cite{HolSpu}, we define $\mathscr U$ to be {\it regular in $X$}, if
\begin{enumerate}[label=(\alph*)]
  \item each $U_\xi$ is open in $X$;

  \item $\emptyset=U_0\subset U_1\subset U_2\subset\dots\subset U_\alpha=X$;

  \item\label{it:c} $U_\gamma=\bigcup_{\xi<\gamma} U_\xi$ for every limit ordinal $\gamma\in[0,\alpha)$.
\end{enumerate}

\begin{proposition}\label{prop:char_frag}
  Let $X$ be a topological space, $(Y,d)$ be a metric space and $\varepsilon>0$. For a map $f:X\to Y$ the following conditions are equivalent:
  \begin{enumerate}
    \item\label{prop:char_frag:it:1}  $f$ is $\varepsilon$-fragmented;

    \item\label{prop:char_frag:it:2} there exists a regular sequence $\mathscr U=(U_\xi:\xi\in[0,\alpha])$ in $X$ such that ${\rm diam} f(U_{\xi+1}\setminus U_\xi)<\varepsilon$ for all $\xi\in[0,\alpha)$.
  \end{enumerate}
\end{proposition}

\begin{proof}
 (\ref{prop:char_frag:it:1})$\Rightarrow$(\ref{prop:char_frag:it:2}) is proved in \cite[Proposition 3.1]{Karlova:Mykhaylyuk:2016:EJMA}.

 (\ref{prop:char_frag:it:2})$\Rightarrow$(\ref{prop:char_frag:it:1}).  We fix a nonempty closed set $F\subseteq X$. Denote $\beta=\min\{\xi\in[0,\alpha]:F\cap U_\xi\ne\emptyset\}$. Property~\ref{it:c} implies that $\beta=\xi+1$ for some $\xi<\alpha$. Then the set $U=U_\beta\cap F$ is open in $F$ and ${\rm diam} f(U)\le{\rm diam}f(U_{\xi+1}\setminus U_\xi)<\varepsilon$.
\end{proof}

If a sequence $\mathscr U$ satisfies condition~(\ref{prop:char_frag:it:2}) of Proposition~\ref{prop:char_frag}, then it is called {\it $\varepsilon$-associated with $f$} and is denoted by $\mathscr U_\varepsilon(f)$.

We say that an $\varepsilon$-fragmented  map $f:X\to Y$ is {\it functionally $\varepsilon$-fragmented} if $\mathscr U_\varepsilon(f)$ can be chosen such that every set $U_\xi$ is functionally open in $X$. Further, $f$ is {\it functionally $\varepsilon$-countably fragmented} if $\mathscr U_\varepsilon(f)$ can be chosen to be countable and $f$ is {\it functionally countably fragmented} if $f$ is functionally $\varepsilon$-countably fragmented for all $\varepsilon>0$.

Evident connections between     kinds of fragmentability and its analogs are gathered in the following diagram.

\bigskip
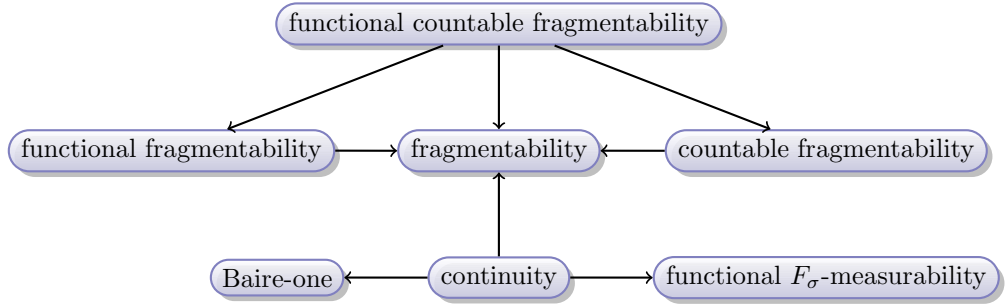
\begin{center}
 \begin{tikzpicture}[thick,
             to/.style={->,>=stealth',shorten >=2pt,thick,font=\sffamily\footnotesize},
             toto/.style={<->,>=stealth',shorten >=2pt,thick,font=\sffamily\footnotesize},]

  \tikzstyle{format} = [rounded rectangle,
                      thick,
                      text centered,
                      draw=blue!50!black!50,
                      top color=white,
                      bottom color=blue!50!black!20,
                      drop shadow]

\node[format,anchor=center]  (FCF) {functional countable fragmentability};
\node[format,below=1.4cm  of FCF,anchor=center]  (F)  {fragmentability};
\node[format,right=3cm  of F, anchor=center] (CF) {countable fragmentability};
\node[format,left=3cm  of F, anchor=center]  (FF) {functional  fragmentability};
\node[format,below=1.4cm of F, anchor=center] (C) {continuity};
\node[format,left=2cm  of C, anchor=center]  (B1) {Baire-one};
\node[format,below=1.4cm  of CF, anchor=center]  (H1) {functional $F_\sigma$-measurability};

\path[->] (FCF) edge (FF);
\path[->] (FCF) edge (F);
\path[->] (FCF) edge (CF);
\path[->] (CF) edge (F);
\path[->] (FF) edge (F);
\path[->] (C) edge (F);
\path[->] (C) edge (B1);
\path[->] (C) edge (H1);
\end{tikzpicture}
\end{center}
\bigskip

Notice that none of the inverse implications is true.
\begin{remark}{\rm
  \begin{enumerate}[label=(\alph*)]
    \item If $X$ is hereditarily Baire, then every Baire-one map $f:X\to (Y,d)$ is barely continuous (i.e., for every nonempty closed set $F\subseteq X$ the restriction $f|_F$ has a point of continuity) and, hence, is fragmented (see \cite[31.X]{Kur:Top:Eng1}).

    \item If $X$ is a  paracompact space in which every closed set is $G_\delta$, then every fragmented map $f:X\to (Y,d)$ is Baire-one in the case either ${\rm dim} X=0$, or $Y$ is a metric contractible locally path-connected space~\cite{Karlova:Mykhaylyuk:Comp,Karlova:Mykhaylyuk:2016:EJMA}.

    \item Let $X=\mathbb R$ be endowed with the topology generated by the discrete metric $d(x,y)=1$ if $x\ne y$, and $d(x,y)=0$ if $x=y$. Then the identical map $f:X\to X$ is continuous, but is not countably fragmented.
  \end{enumerate}}
\end{remark}
For a deeper discussion of properties and applications of fragmented maps and their analogs we refer the reader to~\cite{AgCN,BanakhBokalo:weak,JOPV,Kum,Spu}.

\begin{proposition}\label{prop:ex_of_countablyfrag}
Let $X$ be a topological space, $(Y,d)$ be a metric space, $\varepsilon>0$ and $f:X\to Y$ be a map. If one of the following conditions hold
\begin{enumerate}
\item\label{it:prop:ex_of_countablyfrag:1} $Y$ is separable and $f$ is continuous,

\item\label{it:prop:ex_of_countablyfrag:2} $X$ is metrizable separable and $f$ is fragmented,

\item\label{it:prop:ex_of_countablyfrag:3} $X$ is compact and $f\in{\rm B}_1(X,Y)$,
\end{enumerate}
then $f$ is functionally countably fragmented.
\end{proposition}

\begin{proof} Fix $\varepsilon>0$.

\textup{(\ref{it:prop:ex_of_countablyfrag:1})} Choose a  covering $(B_n:n\in\mathbb N)$ of $Y$ by open balls of diameters $<\varepsilon$. Let $U_0=\emptyset$, $U_n=f^{-1}(\bigcup_{k\leq n} B_k)$ for every $n\in\mathbb N$ and $U_{\omega_0}=\bigcup_{n=0}^\infty U_n$. Then the sequence $(U_\xi:\xi\in[0,\omega_0])$ is $\varepsilon$-associated with  $f$.

 \textup{(\ref{it:prop:ex_of_countablyfrag:2})} Notice that any strictly increasing well-ordered chain of open sets in $X$ is at most countable and every open set in $X$ is functionally open.

 \textup{(\ref{it:prop:ex_of_countablyfrag:3})}  By \cite[Proposition 7.1]{Karlova:Mykhaylyuk:2016:EJMA} there exist a metrizable compact space   $Z$, a continuous function  $\varphi:X\to Z$ and a function  $g\in{\rm B}_1(Z,Y)$ such that $f=g\circ \varphi$. Then $g$ is functionally $\varepsilon$-countably fragmented by condition (\ref{it:prop:ex_of_countablyfrag:2}) of the theorem. It is easy to see that $f$ is functionally $\varepsilon$-countably fragmented too.
\end{proof}

\begin{lemma}\label{lem:unif_conv_ext}
Let $X$ be a topological space,  $E\subseteq X$ and $f\in {\rm B}_1(E,\mathbb R)$. If there exists a sequence of functions $f_n\in {\rm B}_1(X,\mathbb R)$ such that $(f_n)_{n=1}^\infty$ converges uniformly to $f$ on $E$, then $f$ can be extended to a function $g\in {\rm B}_1(X,\mathbb R)$.
\end{lemma}

\begin{proof} Without loss of generality we may assume that $f_0(x)=0$ for all $x\in E$ and
$$
|f_{n}(x)-f_{n-1}(x)|\le \frac{1}{2^{n-1}}
$$
for all $n\in\mathbb N$ and $x\in E$. Now we put
$$
g_n(x)=\max\{\min\{(f_{n}(x)-f_{n-1}(x)),2^{-n+1}\},-2^{-n+1}\}
$$
and notice that  $g_n\in{\rm B}_1(X,\mathbb R)$. Moreover,  the series  $\sum_{n=1}^\infty g_n(x)$ is uniformly convergent on $X$ for a function  $g\in{\rm B}_1(X,\mathbb R)$. Then
$g$ is the required extension of $f$.
\end{proof}

Recall that a subspace $E$ of a topological space $X$ is   {\it $z$-embedded in $X$} if for any zero set $F$ in $E$ there exists a zero set $H$ in $X$ such that $H\cap E=F$;  {\it $C^*$-embedded in $X$} if any bounded continuous function $f$ on $E$ can be extended to a continuous function on $X$.

\begin{proposition}\label{prop:ext_countable_frag_ness}
  Let $E$ be a $z$-embedded subspace of a completely regular space $X$ and $f:E\to \mathbb R$ be a functionally countably fragmented  function. Then  $f$ can be extended to a functionally countably fragmented function $g\in {\rm B}_1(X,\mathbb R)$.
  \end{proposition}

  \begin{proof} Let us observe that we may assume the space $X$ to be compact. Indeed, $E$ is $z$-embedded in $\beta X$, since $X$ is $C^*$-embedded in $\beta X$~\cite[Theorem 3.6.1]{Eng-eng}, and if we can extend $f$ to a  functionally countably fragmented function $h\in {\rm B}_1(\beta X,\mathbb R)$, then the restriction $g=h|_E$ is a  functionally  countably fragmented extension of $f$ on $X$ and $g\in {\rm B}_1(X,\mathbb R)$.

  Fix $n\in\mathbb N$ and consider  $\frac 1n$-associated with $f$ sequence $\mathscr U=(U_\xi:\xi\leq\alpha)$. Without loss of the generality we can assume that all sets $U_{\xi+1}\setminus U_\xi$ are nonempty.
   Since $E$ is $z$-embedded in $X$, one can choose a countable family $\mathscr V=(V_\xi:\xi\leq\alpha)$ of functionally open sets in $X$ such that $V_{\xi}\subseteq V_{\eta}$ for all $\xi\le\eta\leq\alpha$, $V_{\xi}\cap E=U_{\xi}$ for every $\xi\leq\alpha$ and $V_\eta=\bigcup_{\xi<\eta}V_\xi$ for every limit ordinal $\eta\leq \alpha$. For every $\xi\in[0,\alpha)$ we take an arbitrary point $y_\xi\in f(U_{\xi+1}\setminus U_\xi)$. Now for every $x\in X$ we put
  $$
  f_n(x)=\left\{\begin{array}{ll}
              y_\xi, &  x\in V_{\xi+1}\setminus V_\xi,\\
              y_0,   & x\in X\setminus V_\alpha.
                \end{array}\right.
   $$
    Observe that $f_n:X\to \mathbb R$ is functionally $F_\sigma$-measurable, since the preimage $f^{-1}(W)$ of any open set $W\subseteq \mathbb R$ is an at most countable union of functionally $F_\sigma$-sets from the system $\{V_{\xi+1}\setminus V_\xi:\xi\in[0,\alpha)\}\cup\{X\setminus V_\alpha\}$. Therefore, $f_n\in {\rm B}_1(X,\mathbb R)$.

  It is easy to see that the sequence $(f_n)_{n=1}^\infty$ is uniformly convergent to $f$ on $E$. Now it follows from Lemma~\ref{lem:unif_conv_ext} that $f$ can be extended to a function $g\in {\rm B}_1(X,\mathbb R)$. According to Proposition \ref{it:prop:ex_of_countablyfrag:2} (3), $g$ is functionally countably fragmented.
  \end{proof}

\begin{corollary}
  Every functionally countably fragmented function $f:X\to\mathbb R$ defined  on a topological space $X$ belongs to the first Baire class.
\end{corollary}

\begin{proof}
For every $n\in\mathbb N$ we choose a $\frac 1n$-associated with $f$ family $\mathscr U_n=(U_{n,\xi}:\xi\leq\alpha_n)$ of functionally open sets $U_{n,\xi}$ and corresponding family $(\varphi_{n,\xi}:\xi\leq\alpha_n)$ of continuous functions $\varphi_{n,\xi}:X\to[0,1]$ such that $U_{n,\xi}=\varphi^{-1}_{n,\xi}((0,1])$. We consider the at most countable set $\Phi=\bigcup_{n=1}^\infty\{\varphi_{n,\xi}:0\leq \xi\leq \alpha_n\}$ and the continuous mapping $\pi:X\to[0,1]^\Phi$, $\pi(x)=(\varphi(x))_{\phi\in \Phi}$.

Show that $f(x)=f(y)$ for every $x,y\in X$ with $\pi(x)=\pi(y)$. Let $x,y\in X$ with $\pi(x)=\pi(y)$. For every $n\in\mathbb N$ we choose $\xi_n\leq \alpha_n$ such that $x\in U_{n,\xi_n+1}\setminus U_{n,\xi_n}$. Then $y\in U_{n,\xi_n+1}\setminus U_{n,\xi_n}$ and
$$
|f(x)-f(y)|\leq {\rm diam} (U_{n,\xi_n+1}\setminus U_{n,\xi_n})\leq \tfrac1n
$$
for every $n\in\mathbb N$. Thus, $f(x)=f(y)$.

Now we consider the function $g:\pi(X)\to\mathbb R$, $g(\pi(x))=f(x)$. Clearly, that every set $\pi(U_{n,\xi})$ is open in the metrizable space $\pi(X)$. Therefore, for every $n\in\mathbb N$ the family $(\pi(U_{n,\xi}):\xi\leq\alpha_n)$ is  $\frac 1n$-associated with $g$. Thus, $g$ is functionally countably fragmented. According to Proposition~\ref{prop:ext_countable_frag_ness}, $g\in {\rm B}_1(\pi(X),\mathbb R)$. Therefore, $f\in {\rm B}_1(X,\mathbb R)$.
\end{proof}

Combining Propositions~\ref{prop:ex_of_countablyfrag} and \ref{prop:ext_countable_frag_ness}  we obtain the following result.

\begin{theorem}\label{thm:main}
 Let  $X$ be a completely regular space. For a Baire-one function $f:X\to \mathbb R$ the following conditions are equivalent:
  \begin{enumerate}
    \item $f$ is functionally countably fragmented;

    \item $f$ can be extended to a Baire-one function on $\beta X$.
  \end{enumerate}
\end{theorem}

\section{A Baire-one bounded function which is not countably fragmented}\label{s33}

\begin{theorem}
  There exists a completely regular scattered (and hence hereditarily Baire) space $X$ and a Baire-one function $f:X\to [0,1]$ which can not be extended to a Baire-one function on $\beta X$.
\end{theorem}

\begin{proof}
\medskip
\underline{\tt Claim 1. Construction of $X$.} Let $\mathbb Q=\{r_n:n\in\mathbb N\}$, $r_n\ne r_m$ for all distinct $n,m\in\mathbb N$ and
\begin{gather*}
\overline{\{r_{2n-1}:n\in\mathbb N\}}=\overline{\{r_{2n}:n\in\mathbb N\}}=\mathbb Q,\\
A=\bigcup_{n=1}^\infty \{r_{2n-1}\}\times [0,1],\quad B=\bigcup_{n=1}^\infty \{r_{2n}\}\times [0,1].
 \end{gather*}
 We consider partitions $\mathscr A=(A_t:t\in[0,1])$ and $\mathscr B=(B_t:t\in[0,1])$ of the sets $A$ and $B$ into everywhere dense  sets $A_t$ and $B_t$, respectively, such that $|A_t|=|B_t|=\mathfrak c$. Moreover, let $[0,1]=\bigsqcup_{\alpha<\omega_1}T_\alpha$ with $|T_\alpha|=\mathfrak c$ for every $\alpha<\omega_1$. For every $\alpha\in[0,\omega_1)$ we put
 $$
Q_\alpha=\left\{\begin{array}{ll}
              \bigsqcup_{t\in T_\alpha}  A_t, & \alpha\,\,\mbox{is even},\\
              \bigsqcup_{t\in T_\alpha}  B_t, & \alpha\,\,\mbox{is odd},
            \end{array}
 \right.
 $$
$$
Q=\bigsqcup_{\alpha<\omega_1}Q_\alpha, \quad X_\alpha=Q_\alpha\times\{\alpha\} \quad\mbox{and}\quad X=\bigsqcup_{\alpha<\omega_1}X_\alpha.
$$

\medskip
\underline{\tt Claim 2. Indexing of $X$.} For every $\alpha\in[0,\omega_1)$ we consider the set
$$
I_\alpha=\{(i_\xi)_{\xi\in[\alpha,\omega_1)}: |\{\xi:i_\xi\ne 0\}|\le\aleph_0\}\subseteq [0,1]^{[\alpha,\omega_1)}
$$
and notice that $|I_\alpha|=\mathfrak c$. Let $\varphi_\alpha:I_\alpha\to T_\alpha$ be a bijection and
$$
X_\alpha=\bigsqcup_{j\in I_{\alpha+1}}X_{\alpha,j},
$$
where
$$
X_{\alpha,j}=\left\{\begin{array}{ll}
              A_{\varphi_{\alpha+1}(j)}\times \{\alpha\}, & \alpha\,\,\mbox{is even},\\
              B_{\varphi_{\alpha+1}(j)}\times \{\alpha\}, & \alpha\,\,\mbox{is odd},
            \end{array}
 \right.
 $$

For all $\xi, \eta\in [0,\omega_1)$ with $\eta>\xi$ and $i\in I_\eta$ we put
\begin{gather*}
J_{\eta,\xi}^{i}=\{j\in I_\xi: j|_{[\eta,\omega_1)}=i\}.
\end{gather*}
In particular, if $\xi=\alpha$, $\eta=\alpha+1$ and $i\in I_{\alpha+1}$, then we denote the set $J_{\alpha+1,\alpha}^i$ simply by $J_{\alpha}^i$.
Notice that $|J_{\alpha}^i|=\mathfrak c$ and we may assume that
$$
X_{\alpha,i}=\{x_j:j\in J_\alpha^i\}.
$$
Then
$$
X_{\alpha}=\{x_i:i\in I_\alpha\},
$$
since $I_\alpha=\bigsqcup_{i\in I_{\alpha+1}}J_\alpha^i$.

\medskip
\underline{\tt Claim 3. Topologization of $X$.}
For all $\alpha\in[1,\omega_1)$, $i\in I_\alpha$ and $x=x_i\in X_\alpha$  we put
\begin{gather*}
L_{<x}=\bigsqcup_{\xi<\alpha} \{x_j\in X_\xi: j\in J_{\alpha,\xi}^i\},\quad L_{\le x}=\{x\}\cup L_{<x}.
\end{gather*}
Notice that for all $x\in X_\alpha$ and $y\in X_\beta$ with $\alpha\leq \beta$ either $L_{\le x}\subseteq L_{\le y}$, or $L_{\le x}\cap L_{\le y}=\emptyset$.

Now we are ready to define a topology $\tau$ on $X$. Each point of $X_0$ is isolated. For any $\alpha\in[1,\omega_1)$ and a point $x\in X_\alpha$  we construct a base $\mathscr U_x$ of $\tau$-open neighborhoods of $x$ in the following way. Take $i\in I_\alpha$ and $q\in Q$ such that $x=x_i=(q,\alpha)\in X_\alpha$. Let $\mathscr V_q$ be a base of clopen neighborhoods of $q$ in the space $Q$ equipped with the topology induced from $\mathbb R^2$. Then we put
$$
\mathscr U_x=\{(V\times [0,\omega_1))\cap (L_{\le x}\setminus  \bigcup_{y\in Y}L_{\le y}):V\in\mathscr V_q\,\,\,\mbox{and}\,\,\, Y\subseteq L_{<x}\,\,\mbox{is finite}\}.
$$

\underline{\tt Claim 4. Complete regularity and scatteredness of $X$.} We show that the space $(X,\tau)$ is completely regular. We prove firstly that every set $L_{\leq x}$ is clopen. Since the inclusion $v\in L_{\leq u}$ implies $L_{\leq v}\subseteq L_{\leq u}$, every set $L_{\leq x}$ is open.
Now let $y\in \overline{L_{\leq x}}$. Then $L_{\leq y}\cap L_{\leq x}\ne\emptyset$. Therefore, $L_{\leq y}\subseteq L_{\leq x}$ or $L_{\leq x}\subseteq L_{\leq y}$. Assume that $y\not\in L_{\leq x}$. Then $x\in L_{< y}$ and for a neighborhood $W=L_{\leq y}\setminus L_{\leq x}$ of $y$ we have $W\cap L_{\leq x}=\emptyset$, which implies a contradiction. Thus, $\overline{L_{\leq x}}=L_{\leq x}$ and the set $L_{\leq x}$ is closed.

Notice that for every clopen in $Q$ set $V$ the set $(V\times [0,\omega_1))\cap X$ is clopen in $X$. Therefore, every $U\in\mathscr U_x$ is clopen in $X$ for every $x\in X$. In particular, $(X,\tau)$ is completely regular.

In order to show that $(X,\tau)$ is scattered we take an arbitrary nonempty set $E\subseteq X$ and denote $\alpha=\min\{\xi\in[0,\omega_1): E\cap X_\xi\ne\emptyset\}$. Then any point $x$ from $E\cap X_\alpha$ is isolated in $E$.

\medskip
\underline{\tt Claim 5. $\overline{X_\alpha}=\bigcup_{\xi\geq\alpha}X_\xi$ for every $\alpha\in [0,\omega_1)$.} It is sufficient to prove that
$$
X_\alpha\subseteq\overline{\bigcup_{\xi<\alpha}X_\xi}
$$
for every $\alpha\in [1,\omega_1)$.
Let $\alpha\in [1,\omega_1)$, $x=(q,\alpha)\in X_\alpha$, $V$ be an open neighborhood of $q$ in $Q$, $Y\subseteq L_{<x}$ be a finite set and
$$
U=(V\times[0,\omega_1))\cap (L_{\leq x}\setminus \bigcup_{y\in Y}L_{\leq y}).
$$
We show that $U\cap (\bigcup_{\xi<\alpha}X_\xi)\ne \emptyset$. For every $y\in Y$ we choose $\beta_y<\alpha$ such that $y\in X_{\beta_y}$. We put $\beta={\rm max}\{\beta_y:y\in Y\}$ and $\gamma=\beta+1$. Since $\beta<\alpha$, $\gamma\leq \alpha$. We choose $i\in I_\alpha$ such that $x=x_i$ and choose $j\in J_\gamma$ such that $j{|_{[\alpha,\omega_1)}}=i$. We consider the set $X_{\beta,j}$. Recall that $X_{\beta,j}=A_t\times\{\beta\}$ or $X_{\beta,j}=B_t\times\{\beta\}$, where $t=\varphi_\gamma(j)$. Therefore, the set
$$
P=\{p\in Q:(p,\beta)\in X_{\beta,j}\}
$$
is dense in $Q$. Thus, the set $P\cap V$ is infinite. Moreover, $|L_{\leq y}\cap X_\beta|\leq 1$ for every $y\in Y$. Hence, the set
$$
S=\{p\in P:(p,\beta)\in \bigcup_{y\in Y}L_{\leq y}\}
$$
is finite. Therefore, the set $(P\cap V)\setminus S$ is infinite, in particular, it is nonempty. We choose a point $p\in (P\cap V)\setminus S$. Then $z=(p,\beta)\in X_{\beta,j}$. Thus, $z\in L_{\leq x}$. Moreover, $z\in V\times [0,\omega_1)$ and $z\not\in \bigcup_{y\in Y}L_{\leq y}$. Thus, $z\in U$. Since $X_{\beta,j}\subseteq X_\beta$, $z\in \bigcup_{\xi<\alpha}X_\xi$. Therefore, $U_x\cap \bigcup_{\xi<\alpha}X_\xi\ne \emptyset$ for every $U_x\in {\mathscr U}_x$ and $x\in \overline{\bigcup_{\xi<\alpha}X_\xi}$.

\medskip
\underline{\tt Claim 6. Construction of a Baire-one function $f$.}
We put
$$
C=\bigsqcup_{\xi<\omega_1,\, \xi \small{\mbox{\,\,is even}}} X_\alpha
$$
and show that the function $f:X\to [0,1]$,
$$
f=\chi_{C},
$$
belongs to the first Baire class.

Consider a mapping $\pi:X\to \mathbb Q$, $\pi(x)=r$ if $x=(q,\alpha)$ and $q=(r,t)$ for some $t\in [0,1]$ and $\alpha<\omega_1$. The mapping $\pi$ is continuous, because for every open in $\mathbb Q$ set $V$ the set
$$\pi^{-1}(V)=(V\times [0,1]\times[0,\omega_1))\cap X$$
is open in $X$. Clearly, the function $g:\mathbb Q\to\mathbb R$,
$$
g(t)=\left\{\begin{array}{ll}
             1, & t=r_{2n-1},\\
              0, & t=r_{2n},
            \end{array}
 \right.
 $$
belongs to the first Baire class. Therefore, the function $f(x)=g(\pi(x))$ belongs to the first Baire class too.

\medskip
\underline{\tt Claim 7. The function $f$ is not countably fragmented.} Finally, we prove that $f$ is not countably fragmented. Assume the contrary and take a countable sequence $\mathscr U=(U_\xi:\xi<\alpha)$  of functionally open sets such that $\mathscr U$ is $\frac 12$-associated with $f$.

We show that $U_\beta\subseteq \bigcup_{\xi<\beta}X_\xi$ for every $\beta \leq \alpha$. We will argue  by induction on $\beta$. For $\beta=0$ the assertion is obvious. Assume that the inclusion is valid for all $\beta<\gamma\leq \alpha$. If $\gamma$ is a limit ordinal, then
$$
U_\gamma=\bigcup_{\xi<\gamma}U_\xi\subseteq \bigcup_{\xi<\gamma}\bigcup_{\eta<\xi}X_\eta=\bigcup_{\xi<\gamma}X_\xi.
$$
Now let $\gamma=\delta+1$. Suppose that there exists $x\in U_\gamma\setminus(\bigcup_{\xi\leq\delta}X_\xi)$. Notice that $x\in \overline{X_\gamma}\subseteq \overline{X_\delta}$ according to Claim 6. Therefore, there exist $z_1\in U_\gamma\cap X_\gamma$ and $z_2\in U_\gamma\cap X_\delta$. According to the inductive assumption, we have $U_\delta\subseteq \bigcup_{\xi<\delta}X_\xi$. Therefore, $z_1,z_2\subseteq U_\gamma\setminus U_\delta=U_{\delta+1}\setminus U_\delta$. Thus, we have
$$
1=|f(z_1)-f(z_2)|\leq {\rm diam}(U_{\delta+1}\setminus U_\delta)<\tfrac12,
$$
 a contradiction.

Theorem~\ref{thm:main} implies that $f$ can not be extended to a Baire-one function on $\beta X$.
\end{proof}

\end{document}